\documentclass[11pt]{amsart}
\usepackage{amsfonts,amssymb,amsthm}
\usepackage{amsmath,amscd}
\usepackage{pstricks}
\usepackage{pstricks,pst-node}
\usepackage{mathrsfs}
\usepackage[all]{xy}
\usepackage{hyperref}

\DeclareMathAlphabet{\mathpzc}{OT1}{pzc}{m}{it}

\theoremstyle{plain}
\newtheorem{Thm}{Theorem}[section]
\newtheorem{Prop}[Thm]{Proposition}

\newtheorem{Lem}[Thm]{Lemma}
\newtheorem{Coro}[Thm]{Corollary}
\theoremstyle{definition}
\newtheorem{Def}[Thm]{Definition}

\newtheorem{Rems}[Thm]{Remarks}
\numberwithin{equation}{section}

\def\bffkk{\boldsymbol{\frak k}}

\newcommand{\bfi}{{\mathbf{i}}}

\newcommand{\bfP}{{\mathbf{P}}}
\newcommand{\bfQ}{{\mathbf{Q}}}

\def\fS{{\frak S}}

\newcommand{\ms}{\mathscr}

\newcommand{\msP}{\mathscr P}

\newcommand{\sfz}{{\mathsf z}}

\newcommand{\mpZ}{\mathpzc Z}

\newcommand{\mpf}{\mathpzc f}

\def\sF{{\mathcal F}}
\def\sG{{\mathsf G}}
\def\sGvep{{\mathsf G}_{\varepsilon}}
\def\sH{{\mathcal H}}

\def\sP{{\mathcal P}}

\newcommand{\mbnn}{\mathbb N^{n}}
\newcommand{\mbn}{\mathbb N}

\newcommand{\mbc}{\mathbb C}
\newcommand{\mbz}{\mathbb Z}

\newcommand{\ttk}{\mathtt{k}}
\newcommand{\tth}{\mathtt{h}}

\newcommand{\ttx}{\mathtt{x}}
\newcommand{\ttg}{\mathtt{g}}

\newcommand{\End}{\operatorname{End}}

\newcommand{\la}{{\lambda}}
\newcommand{\La}{\Lambda}
\newcommand{\ga}{{\gamma}}

\newcommand{\dt}{\delta}
\newcommand{\Dt}{\Delta}

\newcommand{\Og}{\Omega}
\newcommand{\og}{\omega}

\newcommand{\vi}{\varphi}
\newcommand{\vep}{\varepsilon}
\newcommand{\ep}{\epsilon}
\newcommand{\al}{\alpha}

\newcommand{\sg}{\sigma}

\newcommand{\ol}{\overline}

\newcommand{\Lcp}{\bar L}
\newcommand{\Vcp}{\bar V}
\newcommand{\Mcp}{\bar M}
\newcommand{\Icp}{\bar I}

\def\ggp#1#2{\left[\kern-3.2pt\left[{#1\atop #2}\right]\kern-3.2pt\right]}

\def\hmod{{\text-}{\mathsf{mod}}}

\def\leq{\leqslant}\def\geq{\geqslant}
\def\le{\leqslant}

\newcommand{\bop}{\bigoplus}

\newcommand{\ot}{\otimes}

\newcommand{\han}{\subseteq}

\newcommand{\h}{\widehat}
\newcommand{\ti}{\widetilde}
\newcommand{\tila}{\widetilde\la}
\newcommand{\timu}{\widetilde\mu}
\newcommand{\tibfQ}{{\widetilde\bfQ}}

\newcommand{\Lanr}{\Lambda(n,r)}

\newcommand{\lra}{\longrightarrow}
\newcommand{\ra}{\rightarrow}

\newcommand{\zr}{\zeta_r}

\newcommand{\mnmod}{\!\!\!\mod\!}

\newcommand{\vtg}{{\!\vartriangle\!}}

\def\ttv{v}

\newcommand{\afHrv}{\sH_{\vtg}(r)_v}
\newcommand{\afHrvep}{\sH_{\vtg}(r)_{\varepsilon}}
\newcommand{\afHrZ}{{{\sH}_{\vtg}(r)_{{\mathpzc Z}}}}

\newcommand{\afUglv}{U_{v}(\widehat{\frak{gl}}_n)}
\newcommand{\afUglZ}{U_{\mpZ}(\widehat{\frak{gl}}_n)}
\newcommand{\afUglvep}{U_{\vep}(\widehat{\frak{gl}}_n)}

\newcommand{\fSr}{\fS_r}

\newcommand{\afSrZ}{{\mathcal S}_{\vtg}(n,r)_{\mathpzc Z}}
\newcommand{\afSrv}{{\mathcal S}_{\vtg}(n,r)_{v}}
\newcommand{\afSrz}{{\mathcal S}_{\vtg}(n,r)_{t}}
\newcommand{\afSNrv}{{\mathcal S}_{\vtg}(N,r)_v}
\newcommand{\afSNrvep}{{\mathcal S}_{\vtg}(N,r)_{\vep}}

\newcommand{\afSNrZ}{{\mathcal S}_{\vtg}(N,r)_{\mathpzc Z}}
\newcommand{\afSrvep}{{\mathcal S}_{\vtg}(n,r)_{\vep}}

\newcommand{\bfone}{{\mathfrak l}}

\newcommand{\afUslv}{U_{v}(\widehat{\frak{sl}}_n)}
\newcommand{\afUslZ}{U_{\mathpzc Z}(\widehat{\frak{sl}}_n)}
\newcommand{\afUslZpm}{U_{\mathpzc Z}^\pm(\widehat{\frak{sl}}_n)}
\newcommand{\afUslZp}{U_{\mathpzc Z}^+(\widehat{\frak{sl}}_n)}
\newcommand{\afUslZm}{U_{\mathpzc Z}^-(\widehat{\frak{sl}}_n)}
\newcommand{\afUslZz}{U_{\mathpzc Z}^0(\widehat{\frak{sl}}_n)}
\newcommand{\afUslvep}{U_{\varepsilon}(\widehat{\frak{sl}}_n)}

\newcommand{\afUnrZ}{U_{\vtg}(n,r)_{\mathpzc Z}}

\newcommand{\afUnZ}{U_{\vtg}(n)_{\mathpzc Z}}
\newcommand{\afUnv}{U_{\vtg}(n)_v}
\newcommand{\afUnrv}{U_{\vtg}(n,r)_v}
\newcommand{\afUnrz}{U_{\vtg}(n,r)_t}
\newcommand{\afUnvep}{U_{\vtg}(n)_{\varepsilon}}
\newcommand{\afUnrvep}{U_{\vtg}(n,r)_{\varepsilon}}
\newcommand{\afUNrvep}{U_{\vtg}(N,r)_{\varepsilon}}

\newcommand{\Pn}{\mathpzc P(n)}
\newcommand{\Qn}{\mathpzc Q(n)}
\newcommand{\Qnr}{\mathpzc Q(n,r)}
\newcommand{\tiQnr}{\ti{\mathpzc Q}(n,r)}
\newcommand{\QNr}{\mathpzc Q(N,r)}
\newcommand{\Qnv}{\mathpzc Q(n)_{v}}

\newcommand{\mbcv}{\mathbb C(v)}

\newcommand{\OgC}{\Og_{\mathbb C}}
\newcommand{\Ogv}{\Og_{v}}
\newcommand{\OgZ}{\Og_{\mathpzc Z}}
\newcommand{\Ogvep}{\Og_{\varepsilon}}

\begin{document}
\title{Affine quantum Schur algebras at roots of unity}
\author{Qiang Fu}
\address{Department of Mathematics, Tongji University, Shanghai, 200092, China.}
\email{q.fu@hotmail.com}


\thanks{Supported by the National Natural Science Foundation
of China, the Program NCET, Fok Ying Tung Education Foundation
 and the Fundamental Research Funds for the Central Universities}

\begin{abstract}
We will classify finite dimensional irreducible modules for affine quantum Schur algebras at roots of unity and generalize \cite[(6.5f) and (6.5g)]{Gr80} to the affine case in this paper.
\end{abstract}
 \sloppy \maketitle
\section{Introduction}
Finite dimensional irreducible modules for quantum affine algebras $U_v(\h{\frak{g}})$  were classified by Chari--Pressley when $v$ is not a root of unity in \cite{CP91,CPbk,CP95}. In the case where $v$ is an odd root of unity,  finite dimensional irreducible modules for $U_v(\h{\frak{g}})$ were classified by  Chari--Pressley in \cite{CP97}. Frenkel--Mukhin \cite{FM2} generalize Chari--Pressley's result to all roots of unity case.


Finite dimensional polynomial irreducible modules for quantum affine $\frak{gl}_n$ were classified by Frenkel--Mukhin in \cite{FM}. The representation theory of quantum affine $\frak{gl}_n$ is closely related to that of affine quantum Schur algebras. The affine quantum Schur algebra  was introduced by Ginzburg--Vasserot and Lusztig (cf. \cite{GV,Lu99}), which uses cyclic flags and convolution. The algebraic definition of affine quantum Schur algebras was given by R. Green in   \cite{Gr99}. Varagnolo--Vasserot \cite{VV99} proved that the two definitions  of affine quantum Schur algebras are equivalent.



Let $\afUnv$ be the quantum affine algebra over $\mbc(v)$ considered by Lusztig in \cite{Lu99}, where $v$ is an indeterminate. Note that the algebra $\afUnv$ is a proper subalgebra of quantum affine $\frak{gl}_n$. Let $\zeta_r$ be the natural algebra homomorphism from $\afUnv$ to the affine quantum Schur algebra $\afSrv$ over $\mbc(v)$ (see \cite[7.7]{Lu99}). Let $\afUnrv=\zeta_r(\afUnv)$. According to \cite[8.2]{Lu99}, the equality $\afUnrv=\afSrv$ holds if and only if $n>r$.

Let $\afSrZ$ be the affine quantum  Schur algebras over $\mpZ$ and let $\afUnrZ$ be the $\mpZ$-form of $\afUnrv$, where $\mpZ=\mbc[v,v^{-1}]$.  For any $t\in\mbc^*$, let $\afSrz=\afSrZ\ot_\mpZ\mbc$ and $\afUnrz=\afUnrZ\ot_\mpZ\mbc$, where $\mbc$ is regarded as a $\mpZ$-module by specializing $v$ to $t$. Finite dimensional irreducible modules for $\afSrz$ and $\afUnrz$ were classified in \cite{DDF} when $t$ is not a root of unity. In this paper, we will classify finite dimensional irreducible modules for the two algebras $\afSrvep$ and  $\afUnrvep$ in the case where $\vep\in\mbc$ is a primitive $l'$-th root of $1$.
Furthermore, we will generalize \cite[(6.5f) and (6.5g)]{Gr80} to the affine case.

We organize this paper as follows. We recall the definition of quantum affine $\frak{gl}_n$, the affine quantum Schur algebra $\afSrvep$ and the algebra $\afUnrvep$ in \S2, where $\vep\in\mbc$ is a primitive $l'$-th root of $1$. In \S3, we will construct a functor  $\sGvep:\afSNrvep\hmod\ra\afSrvep\hmod$, where $\afSrvep\hmod$ is the category of finite dimensional $\afSrvep$-modules. Furthermore we will generalize \cite[(6.5f)]{Gr80} to the affine case in Proposition \ref{prop1 of sGvep} and Proposition \ref{prop2 of sGvep}. Finally we
will classify finite dimensional irreducible $\afUnrvep$-modules in Theorem \ref{classification Uvtg vep(n,r)}, and use  Proposition \ref{prop1 of sGvep}, Proposition \ref{prop2 of sGvep} and Theorem \ref{classification Uvtg vep(n,r)} to classify finite dimensional irreducible modules for the affine quantum Schur algebra $\afSrvep$ in Theorem \ref{classification afSrvep}. In addition,
we will prove in Proposition \ref{category equivalence of affine quantum Schur algebras} that the functor $\sGvep:\afSNrvep\hmod\ra\afSrvep\hmod$
is an equivalence of categories in the case where $N\geq n\geq r$.
This result is the affine version of \cite[(6.5g)]{Gr80}.

Throughout, let $v$ be an indeterminate and let $\mpZ=\mbc[v,v^{-1}]$.
For $c\in\mbz$ and $t\in\mbn$ let
$$[c]_v=\frac{v^c-v^{-c}}{v-v^{-1}},\quad [t]_v^{!}=\prod_{1\leq i\leq t}[i]_v,\quad\bigg[{c\atop t}\bigg]_v=\prod_{s=1}^t\frac{v^{c-s+1}-v^{-c+s-1}}{v^s-v^{-s}}\in\mpZ.$$
Let $\vep\in\mbc$ be a primitive $l'$-th root of $1$.
Let $l\geq 1$ be defined by
$$l=
\begin{cases}
l'&\text{if $l'$ is
odd},\\
l'/2&\text{if $l'$ is even}.
\end{cases}$$
We make  $\mbc$ into a $\mpZ$-module by specializing $v$ to $\vep$. When $v$ is specialized to $\vep$, $[c]_v$, $[t]_v^!$ and
$\big[{c\atop t}\big]_v$
specialize to complex numbers $[c]_\vep$, $[t]_\vep^!$ and
$\big[{c\atop t}\big]_\vep$ respectively.

\section{Quantum affine $\frak{gl}_n$ and affine quantum Schur algebras}

\subsection{The algebras $\afUglv$, $\afUslv$ and $\afUnv$}

We recall the Drinfeld's new realization of quantum affine $\frak{gl}_n$ as follows (cf. \cite{FM}).
\begin{Def}\label{QLA}
 The {\it quantum loop algebra} $\afUglv$  (or {\it
quantum affine $\mathfrak {gl}_n$})  is the $\mbc(v)$-algebra generated by $\ttx^\pm_{i,s}$
($1\leq i<n$, $s\in\mbz$), $\ttk_i^{\pm1}$ and $\ttg_{i,t}$ ($1\leq
i\leq n$, $t\in\mbz\backslash\{0\}$) with the following relations:
\begin{itemize}
 \item[(QLA1)] $\ttk_i\ttk_i^{-1}=1=\ttk_i^{-1}\ttk_i,\,\;[\ttk_i,\ttk_j]=0$,
 \item[(QLA2)]
 $\ttk_i\ttx^\pm_{j,s}=\ttv^{\pm(\dt_{i,j}-\dt_{i,j+1})}\ttx^\pm_{j,s}\ttk_i,\;
               [\ttk_i,\ttg_{j,s}]=0$,
 \item[(QLA3)] $[\ttg_{i,s},\ttx^\pm_{j,t}]
               =\begin{cases}0,\;\;&\text{if $i\not=j,\,j+1$};\\
                  \pm \ttv^{-js}\frac{[s]_v}{s}\ttx^\pm_{j,s+t},\;\;\;&\text{if $i=j$};\\
                  \mp \ttv^{-js}\frac{[s]_v}{s}\ttx_{j,s+t}^\pm,\;\;\;&\text{if $i=j+1$,}
                \end{cases}$
 \item[(QLA4)] $[\ttg_{i,s},\ttg_{j,t}]=0$,
 \item[(QLA5)]
 $[\ttx_{i,s}^+,\ttx_{j,t}^-]=\dt_{i,j}\frac{\phi^+_{i,s+t}
 -\phi^-_{i,s+t}}{\ttv-\ttv^{-1}}$,
 \item[(QLA6)] $\ttx^\pm_{i,s}\ttx^\pm_{j,t}=\ttx^\pm_{j,t}\ttx^\pm_{i,s}$, for $|i-j|>1$, and
 $[\ttx_{i,s+1}^\pm,\ttx^\pm_{j,t}]_{\ttv^{\pm c_{ij}}}
               =-[\ttx_{j,t+1}^\pm,\ttx^\pm_{i,s}]_{\ttv^{\pm c_{ij}}}$,
 \item[(QLA7)]
 $[\ttx_{i,s}^\pm,[\ttx^\pm_{j,t},\ttx^\pm_{i,p}]_\ttv]_\ttv
 =-[\ttx_{i,p}^\pm,[\ttx^\pm_{j,t},\ttx^\pm_{i,s}]_\ttv]_\ttv\;$ for
 $|i-j|=1$,
\end{itemize}
 where $[x,y]_a=xy-ayx$, and $\phi_{i,s}^\pm$ are defined via the
 generating functions in indeterminate $u$ by
$$\Phi_i^\pm(u):={\ti\ttk}_i^{\pm 1}
\exp\bigl(\pm(\ttv-\ttv^{-1})\sum_{m\geq 1}\tth_{i,\pm m}u^{\pm
m}\bigr)=\sum_{s\geq 0}\phi_{i,\pm s}^\pm u^{\pm s}$$ with
$\ti\ttk_i=\ttk_i\ttk_{i+1}^{-1}$ ($\ttk_{n+1}=\ttk_1$) and $\tth_{i,\pm
m}=\ttv^{\pm(i-1)m}\ttg_{i,\pm m}-\ttv^{\pm(i+1)m}\ttg_{i+1,\pm
m}\,(1\leq i<n).$
\end{Def}

For $1\leq j<n$, let $E_j=\ttx^+_{j,0}$
and $F_j=\ttx^-_{j,0}$ and let
\begin{equation*}
\begin{split}
E_n&= \ttv[\ttx_{n-1,0}^-,[\ttx_{n-2,0}^-,\cdots,
[\ttx_{2,0}^-,\ttx_{1,1}^-]_{\ttv^{-1}}\cdots
]_{\ttv^{-1}}]_{\ttv^{-1}} \ti\ttk_n,\\
F_n&=\ttv^{-1}\ti\ttk_n^{-1}[\cdots[[\ttx_{1,-1}^+,\ttx_{2,0}^+]_\ttv,\ttx_{3,0}^+]_\ttv,
 \cdots,\ttx_{n-1,0}^+]_\ttv.
\end{split}
\end{equation*}
For $s\geq 1$ let
$\sfz^\pm_s=
 s\ttv^{\pm s} \frac1{[s]_\ttv}(\ttg_{1,\pm s}+\cdots+\ttg_{n,\pm s})$.
Then the algebra $\afUglv$ is generated by $E_i$, $F_i$, $\ttk_i^{\pm 1}$ and
$\sfz^\pm_s$ for $1\leq i\leq n$ and $s\geq 1$. Furthermore,
the algebra $\afUglv$ is a Hopf algebra with
comultiplication $\Dt$, counit $\ep$, and antipode $\sg$ defined
by
\begin{eqnarray*}\label{Hopf}
&\Delta(E_i)=E_i\otimes\ti \ttk_i+1\otimes
E_i,\quad\Delta(F_i)=F_i\otimes
1+\ti \ttk_i^{-1}\otimes F_i,&\\
&\Delta(\ttk^{\pm 1}_i)=\ttk^{\pm 1}_i\otimes \ttk^{\pm 1}_i,\quad
\Delta(\sfz_s^\pm)=\sfz_s^\pm\otimes1+1\otimes
\sfz_s^\pm;&\\
&\ep(E_i)=\ep(F_i)=0=\ep(\sfz_s^\pm),
\quad \ep(\ttk_i)=1;&\\
&\sg(E_i)=-E_i\ti\ttk_i^{-1},\quad \sg(F_i)=-\ti\ttk_iF_i,\quad
\sg(\ttk^{\pm 1}_i)=\ttk^{\mp 1}_i,&\\
&\text{and}\;\;\sg(\sfz_s^\pm)=-\sfz_s^\pm,&
\end{eqnarray*}
where $1\leq i\leq n$ and $s\in \mbz^+$ (cf. \cite[4.7]{Be}, \cite[2.2]{Hub2} and \cite[2.3.5 and 4.4.1]{DDF}).

Let $\afUslv$ be
the subalgebra of $\afUglv$ generated by all $\ttx^\pm_{i,s}$, $\ti\ttk_i^{\pm1}$ and $\tth_{i,t}$. Then $\afUslv$ is quantum affine $\frak{sl}_n$ (cf. \cite{Be}).
Let $\afUnv$ be the subalgebra of $\afUglv$ generated by all $\ttx^\pm_{i,s}$, $\ttk_i^{\pm1}$ and $\tth_{i,t}$. The algebra $\afUnv$ is the quantum affine algebra considered by Lusztig in \cite{Lu99}.

For $1\leq i\leq n$ and $s\in\mbz$, define the elements $\ms
Q_{i,s}\in\afUglv$ through the generating functions
\begin{equation*}
\begin{split}
&\quad\qquad\ms Q_i^\pm(u):=\exp\bigg(-\sum_{t\geq
1}\frac{1}{[t]_\ttv}g_{i,\pm t} (\ttv u)^{\pm t}\bigg)=\sum_{s\geq
0}\ms Q_{i,\pm s} u^{\pm s}\in\afUglv[[u,u^{-1}]].
\end{split}
\end{equation*}
Similarly, for $1\leq j\leq n-1$ and $s\in\mbz$, define the elements $\ms
P_{j,s}\in\afUslv$ through the generating functions
\begin{equation*}
\begin{split}
& \ms P_j^\pm(u):=\exp\bigg(-\sum_{t\geq
1}\frac{1}{[t]_\ttv}\tth_{j,\pm t} (\ttv u)^{\pm
t}\bigg)=\sum_{s\geq 0}\ms P_{j,\pm s} u^{\pm
s}\in\afUslv[[u,u^{-1}]].
\end{split}
\end{equation*}
By definition we have
\begin{equation*}
\begin{split}
& \ms P_j^\pm(u)=\frac{\ms Q_j^\pm(u\ttv^{j-1})}{\ms
Q_{j+1}^\pm(u\ttv^{j+1})},
\end{split}
\end{equation*}
for $1\leq j\leq n-1$.

Let $\mpZ=\mbc[v,v^{-1}]$. Following \cite{CP97}, let $\afUslZ$ be the $\mpZ$-subalgebra of $\afUslv$ generated by $(\ttx_{i,s}^\pm)^{(m)}$, $\ti\ttk_i^{\pm 1}$,
$\big[{\ti\ttk_i;0\atop t}\big]$, $\ms P_{i,s}$ ($1\leq i\leq n-1$, $m,t\in\mbn$, $s\in\mbz$), where
$$(\ttx_{i,s}^\pm)^{(m)}=\frac{(\ttx_{i,s}^\pm)^m}{[m]^!}
\text{ and }
\bigg[ {\ti\ttk_i;0 \atop t} \bigg] =
\prod_{s=1}^t \frac
{\ti\ttk_iv^{-s+1}-\ti\ttk_i^{-1}v^{s-1}}
{v^s-v^{-s}}.$$
Let $\afUnZ$ be the $\mpZ$-subalgebra of $\afUnv$ generated by $(\ttx_{i,s}^\pm)^{(m)}$, $\ttk_j^{\pm 1}$,
$\big[{\ttk_j;0\atop t}\big]$, $\ms P_{i,s}$ ($1\leq i\leq n-1$, $1\leq j\leq n$, $m,t\in\mbn$, $s\in\mbz$), where
$$\bigg[ {\ttk_j;0 \atop t} \bigg] =
\prod_{s=1}^t \frac
{\ttk_jv^{-s+1}-\ttk_j^{-1}v^{s-1}}
{v^s-v^{-s}}.$$

\subsection{Affine quantum Schur algebras}

We now recall the definition of affine quantum Schur algebras.
The extended affine Hecke algebra $\afHrZ$ is defined to be the $\mpZ$-algebra generated by
$$T_i,\quad X_j^{\pm 1}(\text{$1\leq i\leq r-1$, $1\leq j\leq r$}),$$
 and relations
$$\aligned
 & (T_i+1)(T_i-\ttv^2)=0,\\
 & T_iT_{i+1}T_i=T_{i+1}T_iT_{i+1},\;\;T_iT_j=T_jT_i\;(|i-j|>1),\\
 & X_iX_i^{-1}=1=X_i^{-1}X_i,\;\; X_iX_j=X_jX_i,\\
 & T_iX_iT_i=\ttv^2 X_{i+1},\;\;  X_jT_i=T_iX_j\;(j\not=i,i+1).
\endaligned$$

Let $\fSr$ be the symmetric group with generators $s_i:=(i,i+1)$ for $1\leq i\leq r-1$.
Let $I(n,r)=\{(i_1,\ldots,i_r)\in\mbz^r\mid 1\leq i_k\leq
n,\,\forall k\}.$
The symmetric group $\fSr$ acts
on the set $I(n,r)$ by place permutation:
\begin{equation*}\label{place permutation}
\bfi w=(i_{w(k)})_{k\in\mbz},\quad\text{
for $\bfi\in I(n,r)$ and $w\in\fSr$.}
\end{equation*}

Let $\OgZ$ be the free $\mpZ$-module with basis $\{\og_i\mid i\in\mathbb Z\}$.  For
 $\bfi=(i_1,\ldots,i_r)\in\mbz^r$, write
$$\og_\bfi=\og_{i_1}\ot\og_{i_2}\ot\cdots\ot \og_{i_r}=\og_{i_1}\og_{i_2}\cdots \og_{i_r}\in\OgZ^{\ot r}.$$
The tensor space $\OgZ^{\ot r}$
admits a right $\afHrZ$-module structure defined by
\begin{equation*}\label{afH action}
\begin{cases}
\og_{\bf i}\cdot X_t^{-1}
=\og_{i_1}\cdots\og_{i_{t-1}}\og_{i_t+n}\og_{i_{t+1}}\cdots\og_{i_r},\qquad \text{ for all }\bfi\in \mbz^r;\\
{\og_{\bf i}\cdot T_k=\left\{\begin{array}{ll} \ttv^2\og_{\bf
i},\;\;&\text{if $i_k=i_{k+1}$;}\\
\ttv\og_{\bfi s_k},\;\;&\text{if $i_k<i_{k+1}$;}\qquad\text{ for all }\bfi\in I(n,r),\\
\ttv\og_{\bfi s_k}+(\ttv^2-1)\og_{\bf i},\;\;&\text{if
$i_{k+1}<i_k$,}
\end{array}\right.}
\end{cases}
\end{equation*}
where $1\leq k\leq r-1$ and $1\le t\le r$ (cf. \cite{VV99}).
The algebra $$\afSrZ:=\End_{\afHrZ}(\OgZ^{\ot r})$$
is called an affine quantum Schur algebra.

Let $\Ogv=\Og_\mpZ\ot_{\mpZ}\mbc(v)$, $\afHrv=\afHrZ\ot_\mpZ\mbc(v)$ and $\afSrv=\afSrZ\ot_\mpZ\mbc(v)$.
The right action of $\afHrZ$ on $\OgZ^{\ot r}$ extends to a right
action of $\afHrv$ on $\Ogv^{\ot r}$.
Then we have the $\mbc(v)$-algebra isomorphism
$$\afSrv\cong \End_{\afHrv}(\Ogv^{\ot
r}).
$$

Let $\vep\in\mbc$ be a primitive $l'$-th root of $1$. We make  $\mbc$ into a $\mpZ$-module by specializing $v$ to $\vep$.
Let $\Ogvep=\Og_\mpZ\ot_{\mpZ}\mbc$, $\afHrvep=\afHrZ\ot_\mpZ\mbc$ and $\afSrvep=\afSrZ\ot_\mpZ\mbc$. The right action of $\afHrZ$ on $\OgZ^{\ot r}$ induces a right
action of $\afHrvep$ on $\Ogvep^{\ot r}$.
Then we have the $\mbc$-algebra isomorphism
$$\afSrvep\cong \End_{\afHrvep}(\Ogvep^{\ot
r}).
$$

\subsection{The algebras $\afUnrv$ and $\afUnrvep$}

The algebras $\afUglv$ and  $\afSrv$ can be related by a surjective algebra homomorphism $\zeta_r$, which we now describe.
The vector space $\Ogv$ is a $\afUglv$-module with the action
\begin{equation*}
\aligned
E_i\cdot \og_s&=\dt_{\ol{i+1},\bar s}\og_{s-1},\quad F_i\cdot \og_s=\dt_{\bar i,\bar
s}\og_{s+1},\quad
\ttk_i^{\pm 1}\cdot \og_s=\ttv^{\pm\dt_{\bar i,\bar s}}\og_s,\\
&\sfz_t^+\cdot\og_s=\og_{s-tn},\quad\text{and }\;\;
\sfz_t^-\cdot\og_s=\og_{s+tn},
\endaligned
\end{equation*}
where $\bar i$ denotes the corresponding integer modulo $n$.
The Hopf algebra structure of $\afUglv$ induces a $\afUglv$-module $\OgC^{\ot r}$. According to \cite[3.5.5 and 4.4.1]{DDF}, the actions of $\afUglv$ and $\afHrv$ on $\Ogv^{\ot r}$ are commute.  Consequently, there is an algebra homomorphism
\begin{equation}\label{zetar}
\zeta_r:\afUglv\ra\afSrv
\end{equation}
It is proved in \cite[3.8.1]{DDF} that $\zr$ is surjective.
Every $\afSrv$-module will be inflated into a $\afUglv$-module
via $\zeta_r$.

Let $\afUnrv=\zeta_r(\afUnv)$, $\afUnrZ=\zeta_r(\afUnZ)$, $\afUnvep=\afUnZ\ot_\mpZ\mbc$ and $\afUnrvep=\afUnrZ\ot_\mpZ\mbc$.
According to \cite[8.2]{Lu99}
the equality $\afUnrZ=\afSrZ$ holds if and only if  $n>r$.
The intersection
cohomology basis for $\afUnrZ$ was constructed in [loc. cit.] and the presentation of $\afUnrv$ was obtained in \cite{DGr,Mc07}.
By restriction, the map $\zeta_r$ induces a surjective map $\zeta_r:\afUnZ\ra\afUnrZ$. Thus, base change induces a surjective algebra homomorphism
\begin{equation}\label{zetarvep}
\zeta_{r,\vep}:=\zeta_r\ot 1:\afUnvep\ra\afUnrvep.
\end{equation}
Every $\afUnrvep$-module will be inflated into a $\afUnvep$-module via $\zeta_{r,\vep}$.
Let $$\afUslvep=\afUslZ\ot_\mpZ\mbc.$$ Then by \cite[9.3]{DFW} we have
\begin{equation}\label{image of afuslvep}
\afUnrZ=\zeta_r(\afUslZ)  \text{ and }\afUnrvep=\zeta_{r,\vep}(\afUslvep).
\end{equation}
If $x$ is any element of $\afUslZ$ (resp. $\afSrZ$), we denote the corresponding element of  $\afUslvep$ (resp. $\afSrvep$) also by $x$.

\section{The functor $\sGvep$}
For any algebra $B$, the category of all finite dimensional
left $B$-modules will be denoted by $B\hmod$.  In this section we will construct a functor $\sGvep:\afSNrvep\hmod\ra\afSrvep\hmod$ and generalize \cite[6.5(f)]{Gr80} to the affine case in Proposition \ref{prop1 of sGvep} and Proposition \ref{prop2 of sGvep}, which will be used in \S 4.

\subsection{The functors $\sG$ and $\sGvep$}

For $1\leq i\leq n$ and $t\in\mbn$ let $\bffkk_i=\zeta_r(\ttk_i)$ and $\big[{\bffkk_i;0\atop t}\big]=\zeta_r(\big[{\ttk_i;0\atop t}\big])$ . Let $\Lanr=\{\la\in\mbnn\mid\sum_{1\leq i\leq n}\la_i=r\}$. For $\la\in\Lanr$ let $\bfone_\la=\big[{\bffkk_1;0\atop \la_1}\big]
\cdots\big[{\bffkk_n;0\atop \la_n}\big]$. According to \cite[3.7.4]{DDF}
we have
\begin{equation}\label{bfone}
1=\sum_{\la\in\Lanr}\bfone_\la,\quad\bfone_\la\bfone_\mu=\dt_{\la,\mu}
\bfone_\la,\quad\bffkk_i=\sum_{\la\in\Lanr}v^{\la_i}\bfone_\la.
\end{equation}

Assume $N\geq n$.
For $\mu\in\Lanr$ let $$\ti\mu=(\mu_1,\cdots,\mu_n,0,\cdots,0)\in\La(N,r).$$
Let $$e=\sum_{\mu\in\Lanr}\bfone_{\timu}\in\afSNrZ.$$ Then $e\afSNrv e\cong\afSrv $.
Consequently, we may identify $e\afSNrv e\hmod$ with $\afSrv\hmod$. With this identification, we
define a functor
\begin{equation*}
\sG=\sG_{N,n}:\afSNrv\hmod\lra\afSrv\hmod,\qquad
V\longmapsto eV.
\end{equation*}
Similarly, we define the functor $\sGvep:\afSNrvep\hmod\lra\afSrvep\hmod$ as follows. We shall denote the image of $e$ in $\afSNrvep$ by the same letter.
Since $e\afSNrvep e\cong\afSrvep $, we may identify $e\afSNrvep e\hmod$ with $\afSrvep\hmod$. Therefore, we may define a functor
\begin{equation*}
\sGvep=\sG_{N,n,\vep}:\afSNrvep\hmod\lra\afSrvep\hmod,\qquad
V\longmapsto eV.
\end{equation*}

\subsection{Properties of the functor $\sG$}

We now recall a result about $\sG$ established in \cite{Fu12}.
Let $\Pi_v$ be the set of polynomials $Q(u)\in\mbcv[u]$ such that the constant term of $Q(u)$ is $1$.
Following \cite{FM}, an $n$-tuple of polynomials
$\bfQ=(Q_1(u),\ldots,Q_n(u))$ with $Q_i(u)\in\Pi_v$ is called {\it
dominant} if
$Q_i(\ttv^{i-1}u)/Q_{i+1}(\ttv^{i+1}u)\in\mbcv[u]$ for $1\leq i\leq n-1$. Let $\Qnv$ be the set of dominant $n$-tuples of polynomials.
Let $$\Qn=\{\bfQ\in\Qnv\mid Q_n(uv^{n-1})\in\mbc[u],\,Q_i(\ttv^{i-1}u)/Q_{i+1}(\ttv^{i+1}u)\in\mbc[u],\,\text{for}\,1\leq i\leq n-1\}.$$
For $\bfQ\in\Qn$ let $\deg\bfQ=(\deg Q_1(u),\cdots,\deg Q_n(u))$.

For $g(u)=\prod_{1\leq i\leq m}(1-a_iu)\in\mbcv[u]$
with constant term $1$ and $a_i\in\mbcv^*$, define
\begin{equation}\label{f^pm(u)}
g^\pm(u)=\prod_{1\leq i\leq m}(1-a_i^{\pm1}u^{\pm1}).
\end{equation}
For $\bfQ=(Q_1(u),\ldots,Q_{n}(u))\in\Qn$, define
$Q_{i,s}\in\mbc$, for $1\leq i\leq n$ and $s\in\mbz$, by the
following formula
$$Q_i^\pm(u)=\sum_{s\geq 0}Q_{i,\pm s}u^{\pm s},$$
where $Q_i^\pm(u)$ is defined using \eqref{f^pm(u)}. Let $I(\bfQ)$
be the left ideal of $\afUglv$ generated by $\ttx_{j,s}^+ ,\quad\ms
Q_{i,s}-Q_{i,s},$ and $\ttk_i-\ttv^{\la_i}$, for $1\leq j\leq n-1$,
$1\leq i\leq n$ and $s\in\mbz$, where $\la_i=\mathrm{deg}Q_i(u)$,
and define
$$M(\bfQ)=\afUglv/I(\bfQ).$$
Then $M(\bfQ)$ has a unique irreducible quotient, denoted by $L(\bfQ)$. By \cite[4.3]{FM}, $L(\bfQ)$ is finite dimensional.
Let $$\Qnr=\bigg\{\bfQ\in\Qn\,\big|\,\sum_{1\leq i\leq n}\deg Q_i(u)=r\bigg\}.$$ Then by \cite[4.6.8]{DDF} $L(\bfQ)$ can be regarded as an irreducible module $\afSrv$ via the map $\zeta_r$ given in \eqref{zetar} for $\bfQ\in\Qnr$.

Assume $N\geq n$. We define an injective map $\bfQ\mapsto\ti\bfQ$ of $\Qnr$ into $\QNr$ as follows.
For $\bfQ=(Q_1(u),\cdots,Q_n(u))\in\Qnr$ we define
\begin{equation}\label{tibfQ}
\ti\bfQ=
(Q_1(u),\cdots,Q_n(u),1,\cdots,1)\in\QNr.
\end{equation}
Then the image of $\Qnr$ under this map is the set
$$\tiQnr=\{\ti\bfQ\mid\bfQ\in\Qnr\}.$$
The following result is established in \cite[4.11]{Fu12}.
\begin{Prop}\label{G(L(tibfQ))}
Assume $N\geq n$. Then  $\sG(L(\tibfQ))\cong L(\bfQ)$ for $\bfQ\in\Qnr$. Furthermore, for $\bfQ'\in\QNr$,  $\sG(L(\bfQ'))\not=0$ if and only if  $\bfQ'\in\tiQnr$.
\end{Prop}

\subsection{The irreducible $\afSrvep$-module $V_\vep(\bfQ)$}

For each $\bfQ\in\Qnr$ we will use $L(\bfQ)$ to construct an irreducible $\afSrvep$-module $V_\vep(\bfQ)$ as below.
For $\la,\mu\in\Lanr$ write $\mu\unlhd\la$ if $\sum_{1\leq j\leq i}\mu_j\leq \sum_{1\leq j\leq i}\la_j$ for $1\leq i\leq n$.
By \eqref{bfone} for $\bfQ\in\Qnr$ we have
$$L(\bfQ)=\bop_{\mu\in\Lanr\atop\mu\unlhd\la}L(\bfQ)_\mu$$
where $L(\bfQ)_\mu=\bfone_\mu L(\bfQ)$ and $\la=\deg\bfQ$.
Clearly, $\dim L(\bfQ)_\la=1$. We choose a nonzero vector $w_\bfQ$ in $L(\bfQ)_\la$ and let $L_\mpZ(\bfQ)=\afSrZ w_\bfQ$.
Then $L_\mpZ(\bfQ)=\oplus_{\mu\in\Lanr}L_\mpZ(\bfQ)_\mu$, where $L_\mpZ(\bfQ)_\mu=\bfone_\mu L_\mpZ(\bfQ)$.

\begin{Lem}\label{G(LZ(tibfQ))}
Assume $N\geq n$. Then $\sG(L_\mpZ(\tibfQ))\cong L_\mpZ(\bfQ)$ for $\bfQ\in\Qnr$.
\end{Lem}
\begin{proof}
By Proposition \ref{G(L(tibfQ))} there exist  an $\afSrv$-algebra isomorphism $f:eL(\tibfQ)\ra L(\bfQ)$.  Let $\la=\deg\bfQ$. By \eqref{bfone} we have
$f(eL(\tibfQ)_{\tila})=f(\bfone_{\ti\la}eL(\tibfQ))=\bfone_{\la}f(eL(\tibfQ))=L(\bfQ)_\la$. This, together with the fact that $\dim L(\bfQ)_\la=\dim L(\tibfQ)_{\ti\la}=1$, implies that there exist $k\not=0\in\mbc(v)$ such that  $f(ew_\tibfQ)=kw_\bfQ$. Furthermore  we have $ew_\tibfQ=e\bfone_{\ti\la} w_\tibfQ=w_\tibfQ$ since $w_\tibfQ\in L(\tibfQ)_{\ti\la}$. Consequently,
$\sG(L_\mpZ(\tibfQ))\cong f(eL_\mpZ(\tibfQ))
=f(e\afSNrZ ew_\tibfQ)=\afSrZ f(w_\tibfQ)=\afSrZ (kw_\bfQ)\cong L_\mpZ(\bfQ)$.
\end{proof}

\begin{Lem}\label{free}
For $\bfQ\in\Qnr$, $L_\mpZ(\bfQ)$ is a free $\mpZ$-module.
\end{Lem}
\begin{proof}
By the proof of \cite[8.2]{CP97} and \cite[2.5]{FM2} we see that $\afUslZ w_\bfQ$ is a free $\mpZ$-module.
If $n>r$ then $\afSrZ=\afUnrZ$ by \cite[8.2]{Lu99}. This shows that $L_\mpZ(\bfQ)=\afUnZ w_\bfQ=\afUslZm w_\bfQ=\afUslZ w_\bfQ$,
where  $\afUslZm$ is the $\mpZ$-subalgebra of $\afUslv$ generated by  $(\ttx_{i,s}^-)^{(m)}$ ($1\leq i\leq n-1$, $m\in\mbn$, $s\in\mbz$). Thus $L_\mpZ(\bfQ)$ is a free $\mpZ$-module. Now we assume $n\leq r$. We choose $N$ such that $N>r$.  Since $N>r$, $L_\mpZ(\tibfQ)$ is a free $\mpZ$-module. This implies that $\sG(L_\mpZ(\tibfQ))$ is a free $\mpZ$-submodule of $L_\mpZ(\tibfQ)$  since $\mpZ$ is a principal ideal domain. It follows from Lemma \ref{G(LZ(tibfQ))} that $L_\mpZ(\bfQ)$ is a free $\mpZ$-module.
\end{proof}

Let $L_\vep(\bfQ)=L_\mpZ(\bfQ)\ot_\mpZ\mbc$. Then $L_\vep(\bfQ)=\bop_{\mu\in\Lanr}L_\vep(\bfQ)_\mu$, where $L_\vep(\bfQ)_\mu=\bfone_\mu L_\vep(\bfQ)$.
\begin{Coro}\label{Coro of free module}
For $\bfQ\in\Qnr$ and $\mu\in\Lanr$ we have $\dim_{\,\mbcv} L(\bfQ)=\dim_{\,\mbc} L_\vep(\bfQ)$ and $\dim_{\,\mbcv} L(\bfQ)_\mu=\dim_{\,\mbc} L_\vep(\bfQ)_\mu$.
\end{Coro}
\begin{proof}
By Lemma \ref{free} we have $L(\bfQ)\cong L_\mpZ(\bfQ)\ot_\mpZ\mbcv$ and $L(\bfQ)_\mu\cong L_\mpZ(\bfQ)_\mu\ot_\mpZ\mbcv$. This  implies that
$\dim_{\,\mbcv} L(\bfQ)=\text{rank}_{\mpZ}L_\mpZ(\bfQ)=\dim_{\,\mbc} L_\vep(\bfQ)$ and $\dim_{\,\mbcv} L(\bfQ)_\mu=\text{rank}_{\mpZ}L_\mpZ(\bfQ)_\mu=\dim_{\,\mbc} L_\vep(\bfQ)_\mu$.
\end{proof}

Let $\bfQ\in\Qnr$. According to Corollary \ref{Coro of free module}, we have $\dim_{\,\mbc} L_\vep(\bfQ)_\la=\dim_{\,\mbcv} L (\bfQ)_\la=1$, where $\la=\deg\bfQ$. Thus $L_\vep(\bfQ)$ has a unique finite dimensional irreducible quotient $\afSrvep$-module, denoted by $V_\vep(\bfQ)$.

\subsection{Properties of the functor $\sGvep$}

We are now ready to prove in Proposition \ref{prop1 of sGvep} and Proposition \ref{prop2 of sGvep} that the functor $\sGvep$ enjoys similar properties as the functor $\sG$. These results is the affine version of \cite[(6.5f)]{Gr80}.

\begin{Prop}\label{prop1 of sGvep}
Assume $N\geq n$.
For $\bfQ\in\QNr$, $\sGvep( V_\vep(\bfQ))\not=0$ if and only if $\deg\bfQ\in\ti\La(n,r)$, where $\ti\La(n,r)=\{\ti\mu=(\mu_1,\cdots,\mu_n,0,\cdots,0)
\mid\mu\in\Lanr\}\han\La(N,r)$.
\end{Prop}
\begin{proof}
Let $\bfQ\in\QNr$ and $\la=\deg\bfQ$. If $\la\in\ti\La(n,r)$, then $e \bfone_\la V_\vep(\bfQ)=\bfone_\la V_\vep(\bfQ) \not=0$, and hence $e V_\vep(\bfQ)\not=0$.  Now we assume $e V_\vep(\bfQ)\not=0$.
Since $1=\sum_{\al\in\La(N,r)}\bfone_\al$,
$$eV_\vep(\bfQ)=e\bop_{\al\in\La(N,r)}\bfone_\al V_\vep(\bfQ)=
\bop_{\al\in\Lanr}\bfone_{\ti\al} V_\vep(\bfQ).
$$ This together with the fact that $eV_\vep(\bfQ)\not=0$ implies that there
exists $\al\in\Lanr$ such that $\bfone_{\ti\al} V_\vep(\bfQ)\not=0$. Since $\bfone_{\ti\al} V_\vep(\bfQ)\not=0$, we have $\ti\al\unlhd\la$ and, hence,
$r=\sum_{1\leq i\leq n}\al_i\leq\sum_{1\leq i\leq n}\la_i\leq r$. Therefore, $\la\in\ti\La(n,r)$, as desired.
\end{proof}
\begin{Prop}\label{prop2 of sGvep}
Assume $N\geq n$,  and let $\bfQ\mapsto\ti\bfQ$ be the injective map
from $\Qnr$ into $\QNr$ given in \eqref{tibfQ}.
Then $\sGvep(L_\vep(\tibfQ))\cong L_\vep(\bfQ)$ and $\sGvep(V_\vep(\tibfQ))\cong V_\vep(\bfQ)$ for $\bfQ\in\Qnr$. In particular we have $\dim_\mbc V_\vep(\tibfQ)_\mu=\dim_\mbc V_\vep(\bfQ)_\mu$ for $\mu\in\Lanr$.
\end{Prop}
\begin{proof}
Let $\bfQ\in\Qnr$. Since $L_\vep(\ti\bfQ)\cong(eL_\mpZ(\ti\bfQ)\ot_\mpZ
\mbc)\bop((1-e)L_\mpZ(\tibfQ)\ot_\mpZ\mbc)$, we conclude that $\sGvep(L_\vep(\tibfQ))\cong\sG(L_\mpZ(\tibfQ))\ot_\mpZ\mbc$.
It follows from Lemma \ref{G(LZ(tibfQ))} that
$\sGvep(L_\vep(\tibfQ))\cong L_\mpZ(\bfQ)\ot_\mpZ\mbc=L_\vep(\bfQ)$. This together with the fact that $V_\vep(\tibfQ)$ is the homomorphic image of $L_\vep(\tibfQ)$, implies that $\sGvep(V_\vep(\tibfQ))$ is the homomorphic image of $L_\vep(\bfQ)$. Furthermore by Proposition \ref{prop1 of sGvep} and \cite[6.2(b)]{Gr80} we conclude that $\sGvep(V_\vep(\tibfQ))$ is an irreducible $\afSrvep$-module. Therefore, $\sGvep(V_\vep(\tibfQ))\cong V_\vep(\bfQ)$, since $V_\vep(\bfQ)$ is the unique irreducible quotient $\afSrvep$-module of $L_\vep(\bfQ)$.
\end{proof}

We end this section with an application of the above results.
Let $S$ be an associative algebra over a field $k$ and let
$\mpf\not=0$ be any idempotent in $S$. If $L$ is a
finite dimensional irreducible $S$-module, then by \cite[6.2(b)]{Gr80},
$\mpf L$ is either zero or is an irreducible $\mpf S\mpf$-module. If $\mpf L\not=0$, then
by \cite[6.6(b)]{Gr80}, we have
\begin{equation}\label{[V:L]}
[V:L]=[\mpf V:\mpf L]
\end{equation}
for any
finite dimensional $S$-module $V$, where $[V:L]$ is the multiplicity of $L$ as a composition factor of $V$. Combining Proposition
\ref{G(L(tibfQ))}, Propostion
\ref{prop2 of sGvep} with \eqref{[V:L]} yields the following result.

\begin{Coro}
Assume $N\geq n$,  and let $\bfQ\mapsto\ti\bfQ$ be the injective map
from $\Qnr$ into $\QNr$ given in \eqref{tibfQ}. Then

(1) $[L_\vep(\tibfQ):V_\vep(\ti{\bfQ'})]=[L_\vep(\bfQ):V_\vep(\bfQ')]$ for $\bfQ,\bfQ'\in\Qnr$.

(2) $[L_\vep(\bfQ):V_\vep(\bfQ')]=0$ for $\bfQ\in\QNr\setminus\tiQnr$ and $\bfQ'\in\tiQnr$.
\end{Coro}

\section{Classification of irreducible modules for $\afUnrvep$ and $\afSrvep$}
Finite dimensional irreducible $\afUslvep$-modules were classified by Chari--Pressley \cite{CP97}
in the case
where $\vep$ is a root of unity of odd order.
Frenkel--Mukhin \cite{FM2}  extend Chari--Pressley's result to all
roots of unity.
We will use these results to classify finite dimensional irreducible $\afUnrvep$-modules in Theorem \ref{classification Uvtg vep(n,r)}. Furthermore, we will use Proposition \ref{prop1 of sGvep}, Proposition \ref{prop2 of sGvep} and Theorem \ref{classification Uvtg vep(n,r)} to classify finite dimensional irreducible $\afSrvep$-modules in  Theorem \ref{classification afSrvep}.
Finally we will use Proposition \ref{prop1 of sGvep} and Theorem \ref{classification afSrvep} to generalize \cite[(6.5g)]{Gr80} to the affine case in Proposition \ref{category equivalence of affine quantum Schur algebras}.

\subsection{A simple lemma}

We need some preparation. First we will generalize \cite[3.3]{Lu89} to all roots of unity in Corollary \ref{m=m'}.
\begin{Lem}\label{m t}
Let $m=m_0+lm_1$, $0\leq m_0 \leq l-1$, $m_1 \in\mbn$. Then $$\bigg[{m\atop t}\bigg]_\vep=\vep^{l(t_1l-t_1m_0-t_1lm_1-t_0m_1)}\bigg[{m_0\atop t_0}\bigg]_\vep\bigg({m_1\atop t_1}\bigg)$$ for $0\leq t\leq m$, where
 $t=t_0+lt_1$ with $0\leq t_0\leq l-1$ and $t_1\in\mbn$.
\end{Lem}
\begin{proof}
Let $X$ be an indeterminate. By the proof of \cite[3.2]{Lu89} we have
\begin{equation}\label{formula X}
\prod_{j=0}^{m-1}(1+v^{2j}X)=\sum_{t=0}^m\bigg[{m\atop t}\bigg]_vv^{t(m-1)}X^t\in\mpZ[X].
\end{equation}
This gives
$$\prod_{j=m_0+sl}^{m_0+(s+1)l-1}(1+\vep^{2j}X)=
\prod_{j=0}^{l-1}(1+\vep^{2j}X)=1+\vep^{l(l-1)}X^l$$
for $s\in\mbz$.
It follows that $$\prod_{j=m_0}^{m-1}(1+\vep^{2j}X)=\prod_{s=0}^{m_1-1}
\bigg(\prod_{j=m_0+sl}^{m_0+(s+1)l-1}(1+\vep^{2j}X)\bigg)
=(1+\vep^{l(l-1)}X^l)^{m_1}.$$
This together with \eqref{formula X}
shows that
\begin{equation*}
\begin{split}
\sum_{t=0}^m\bigg[{m\atop t}\bigg]_\vep\vep^{t(m-1)}
X^t&=(1+\vep^{l(l-1)}X^l)^{m_1}\prod_{j=0}^{m_0-1}(1+\vep^{2j}X)\\
&=\sum_{t_1=0}^{m_1}\bigg({m_1\atop t_1}\bigg)(\vep^{l(l-1)}X^l)^{t_1}
\sum_{t_0=0}^{m_0}\bigg[{m_0\atop t_0}\bigg]_\vep\vep^{t_0(m_0-1)}X^{t_0}\\
&=\sum_{0\leq t\leq m,\,t=t_0+lt_1\atop 0\leq t_0\leq l-1,\,t_1\in\mbn}\bigg[{m_0\atop t_0}\bigg]_\vep\bigg({m_1\atop t_1}\bigg)
\vep^{t_0(m_0-1)+t_1l(l-1)}X^t.
\end{split}
\end{equation*}
Comparing the coefficients of $X^t$ in the above equality, we obtain the desired formula.
\end{proof}

\begin{Coro}\label{m l}
Let $m=m_0+lm_1$, $0\leq m_0\leq l-1$, $m_1\in\mbz$. Then
$$\bigg[{m\atop l}\bigg]_\vep=
\begin{cases}
m_1&\text{if $l'$ is odd}\\
(-1)^{l+m}m_1&\text{if $l'$ is even}
\end{cases}
$$
\end{Coro}
\begin{proof}
In the case where $l'$ is odd, the assertion follows from \cite[3.3(a)]{Lu89}. Now we assume $l'$ is even. Then $l'=2l$. Since $\vep$ is a primitive $l'$-th root of unity, we conclude that $\vep^{l}=-1$. If $m_1\geq 1$, then by Lemma \ref{m t}, we have $\big[{m\atop l}\big]_\vep=m_1\vep^{l(l-m)}=(-1)^{l+m}m_1$. If $m_1=0$ then $\big[{m\atop l}\big]=0$, and hence $\big[{m\atop l}\big]_\vep=0$. If $m_1<0$, then $\big[{m\atop l}\big]_\vep=(-1)^l\big[{-m+l-1\atop l}\big]_\vep=(-1)^{l+1}\vep^{l(m+1)}m_1=(-1)^{l+m}m_1$.
\end{proof}

\begin{Coro}\label{m=m'}
If $m,m'\in\mbz$ satisfy $\vep^{m}=\vep^{m'}$, $\big[{m\atop l}\big]_\vep=\big[{m'\atop l}\big]_\vep$, then $m=m'$.
\end{Coro}
\begin{proof}
In the case where $l'$ is odd, the assertion follows from \cite[3.3(b)]{Lu89}. Now we assume $l'$ is even. Then $l'=2l$. Since $\vep^{m}=\vep^{m'}$, we have $m \equiv m'(\mnmod l')$. Thus we may write $m=a+sl'$ and $m'=a+tl'$, where $0\leq a<l'$ and $s,t\in\mbz$. By Corollary \ref{m l} we have
\begin{equation*}
\bigg[{m\atop l}\bigg]_\vep=
\begin{cases}
(-1)^{l+m}2s&\text{if $0\leq a<l$}\\
(-1)^{l+m}(2s+1)&\text{if $l\leq a<l'$}
\end{cases}
\quad\text{and}\
\bigg[{m'\atop l}\bigg]_\vep=
\begin{cases}
(-1)^{l+m'}2t&\text{if $0\leq a<l$}\\
(-1)^{l+m'}(2t+1)&\text{if $l\leq a<l'$}
\end{cases}
\end{equation*}
Since $m-m'=(s-t)l'$ is even we have $(-1)^{l+m}=(-1)^{l+m'}$. This together with the fact that $\big[{m\atop l}\big]_\vep=\big[{m'\atop l}\big]_\vep$ implies that $s=t$. Consequently, $m=m'$, as desired.
\end{proof}

\subsection{Finite dimensional $\afUslvep$-modules}

We now review the classification theorem of finite dimensional irreducible $\afUslvep$-modules.
Let $\Pn$ be the set of $(n-1)$-tuple polynomials $\bfP=(P_1(u),\cdots,P_{n-1}(u))$ such that $P_i(u)\in\mbc[u]$ and the constant term of $P_i(u)$ is $1$ for $1\leq i\leq n-1$.
For $\bfP\in\Pn$, define $P_{j,s}\in\mbc$,
for $1\leq j\leq n-1$ and $s\in\mbz$, as in $P_j^\pm(u)=\sum_{s\geq
0}P_{j,\pm s} u^{\pm s}$, where $P_j^\pm(u)$ is defined using
\eqref{f^pm(u)}.

For $\bfP\in\Pn$ let $\Icp_\vep(\bfP)$ be the left ideal of $\afUslvep$ generated by
$(\ttx_{j,s}^+)^{(m)} ,\ms P_{j,s}-P_{j,s},$ $\ti\ttk_j-\vep^{\mu_j}$ and $\big[{\ti\ttk_j;0\atop l}\big]_\vep-\big[{\mu_j\atop l}\big]_\vep$ for $1\leq j\leq n-1$, $m\in\mbn$ and $s\in\mbz$, where
$\mu_j=\mathrm{deg}P_j(u)$, and define
$$\Mcp_\vep(\bfP)=\afUslvep/\Icp_\vep(\bfP).$$
Then $\Mcp_\vep(\bfP)$ has a unique irreducible quotient, denoted by
$\Vcp_\vep(\bfP)$.

The following classification theorem of finite dimensional irreducible $\afUslvep$-modules is given in \cite[8.2]{CP97} and \cite[2.4]{FM2}.
\begin{Thm}\label{classification quantum affine sln}
The modules $\Vcp_\vep(\bfP)$ with $\bfP\in\Pn$ are all nonisomorphic finite dimensional irreducible $\afUslvep$-modules of type $1$.
\end{Thm}

\subsection{Classification of finite dimensional irreducible $\afUnrvep$-modules}

Now we are ready to classify finite dimensional irreducible $\afUnrvep$-modules.
Let $\bfP\in\Pn$ and $\la\in\La^+(n,r)$ be such that
$\la_i-\la_{i+1}=\mathrm{deg}P_i(u)$, for $1\leq i\leq n-1$. Define
$$
\Mcp(\bfP,\la)=\afUnv/\Icp(\bfP,\la),
$$
where $\bar I(\bfP,\la)$ is the left ideal of $\afUnv$ generated by
$\ttx^+_{i,s}$, $\msP_{i,s}-P_{i,s}$ and $\ttk_j-\ttv^{\la_j}$ for
$1\leq i\leq n-1$, $s\in\mbz$ and $1\leq j\leq n$. The $\afUnv$-module $\Mcp(\bfP,\la)$
has a unique irreducible quotient $\afUnv$-module, which is denoted by
$\Lcp(\bfP,\la)$. Similarly, we define
$$
\Mcp_\vep(\bfP,\la)=\afUnvep/\Icp_{\vep}(\bfP,\la),
$$
where $\Icp_{\vep}(\bfP,\la)$ is the left ideal of $\afUnvep$ generated by
$(\ttx^+_{i,s})^{(m)}$, $\msP_{i,s}-P_{i,s}$, $\ttk_j-\vep^{\la_j}$ and $\big[{\ttk_j;0\atop l}\big]_\vep-\big[{\la_j \atop l}\big]_\vep$  for
$1\leq i\leq n-1$, $s\in\mbz$, $m\in\mbn$ and $1\leq j\leq n$. The $\afUnvep$-module $\Mcp_\vep(\bfP,\la)$
has a unique irreducible quotient $\afUnvep$-module, which is denoted by
$\bar V_\vep(\bfP,\la)$.


\begin{Lem}\label{Lem1 for classification Uvtg vep(n,r)}
Let $\bfP\in\Pn$ and $\la\in\La^+(n,r)$ be such that
$\la_i-\la_{i+1}=\mathrm{deg}P_i(u)$, for $1\leq i\leq n-1$. Then $\bar V_\vep(\bfP,\la)|_{\afUslvep}\cong\bar V_\vep(\bfP)$ and
$\bar V_\vep(\bfP,\la)$ is a $\afUnrvep$-module via the map $\zeta_{r,\vep}$ given in \eqref{zetarvep}.
\end{Lem}
\begin{proof}

By \cite[4.7.5]{DDF}, $\Lcp(\bfP,\la)$ can be regarded as a $\afUnrv$-module via the map $\zeta_r$ given in \eqref{zetar}.
For $\mu\in\Lanr$ let $\Lcp(\bfP,\la)_\mu=\bfone_\mu\Lcp(\bfP,\la)$.
We choose a nonzero vector $w_0$ in $\Lcp(\bfP,\la)_\la$. Let $\Lcp_{\mpZ}(\bfP,\la)=\afUnrZ w_0$ and $\Lcp_\vep(\bfP,\la)=\Lcp_{\mpZ}(\bfP,\la)\ot_\mpZ\mbc$.
Then $$\Lcp_\mpZ(\bfP,\la)=\bigoplus_{\mu\unlhd\la\atop\mu\in\Lanr}
\Lcp_\mpZ(\bfP,\la)_\mu\ \text{and}\
\Lcp_\vep(\bfP,\la)=\bigoplus_{\mu\unlhd\la\atop\mu\in\Lanr}
\Lcp_\vep(\bfP,\la)_\mu
$$
where $\Lcp_\mpZ(\bfP,\la)_\mu=\bfone_\mu\Lcp_\vep(\bfP,\la)$ and $\Lcp_\vep(\bfP,\la)_\mu=\bfone_\mu\Lcp_\vep(\bfP,\la)$.

By the proof of \cite[8.2]{CP97} and \cite[2.5]{FM2} we see that $\afUslZ w_0$ is a free $\mpZ$-module. This implies that $\Lcp_{\mpZ}(\bfP,\la)$ is a free $\mpZ$-module,
since $\Lcp_{\mpZ}(\bfP,\la)=\afUnZ w_0=\afUslZm w_0=\afUslZ w_0$. It follows that $\Lcp_{\mpZ}(\bfP,\la)_\mu$ is a free $\mpZ$-module and hence $\dim_{\,\mbc} \Lcp_\vep(\bfP,\la)_\mu=\dim_{\,\mbcv}\Lcp(\bfP,\la)_\mu$ for $\mu\in\Lanr$. In particular, $\dim_{\,\mbc} \Lcp_\vep(\bfP,\la)_\la=\dim_{\,\mbcv} \Lcp(\bfP,\la)_\la=1$. Therefore, $\Lcp_\vep(\bfP,\la)$ has a unique irreducible quotient $\afUnrvep$-module, denoted by $\bar V_\vep'(\bfP,\la)$. Since $\Vcp_\vep'(\bfP,\la)$ is a homomorphic image of $\bar M_\vep(\bfP,\la)$ we conclude that $\bar V_\vep (\bfP,\la)\cong\bar V_\vep'(\bfP,\la)$ is a $\afUnrvep$-module.
It follows from \eqref{image of afuslvep} that
$\bar V_\vep(\bfP,\la)|_{\afUslvep}$ is irreducible, and hence
$\bar V_\vep(\bfP,\la)|_{\afUslvep}\cong\bar V_\vep(\bfP)$.
\end{proof}

\begin{Lem}\label{Lem2 for classification Uvtg vep(n,r)}
Let $V$ be a finite dimensional irreducible $\afUnrvep$-module.
Then there
exist $\bfP=(P_1(u),\ldots,P_{n-1}(u))\in\sP(n)$ and $\la\in\La^+(n,r)$ with $\la_i-\la_{i+1}
=\mathrm{deg}P_i(u)$, for all $1\leq i\leq n-1$, such that $V\cong
\Vcp_\vep(\bfP,\la)$.
\end{Lem}
\begin{proof}
According to \eqref{image of afuslvep}, we see that $V$ is an irreducible $U_\vep(\widehat{\frak{sl}}_n)$-module. It follows from Theorem \ref{classification quantum affine sln} that there exist $\bfP\in\Pn$ such that $V\cong\bar V_\vep(\bfP)$. Thus there exist $w_0\not=0\in V$ such that
$$(\ttx^+_{i,s})^{(m)}w_0=0, \msP_{i,s}w_0=P_{i,s}w_0,\; \ti\ttk_iw_0=\vep^{\mu_i}w_0,\;\text{ and }\;\bigg[{\ti\ttk_i;0\atop l}\bigg] w_0=\bigg[{\mu_i\atop l}\bigg]_\vep w_0,$$ for
all $1\leq i\leq n-1$ and $s\in\mbz$, where
$\mu_i=\mathrm{deg}P_i(u)$. Since
 $\sum_{\nu\in\Lanr}\bfone_\nu
w_0=w_0\not=0$, there exists $\la\in\Lanr$ such that
$\bfone_\la w_0\not=0$.

Clearly, we have
$$(\ttx^+_{i,s})^{(m)}\bfone_\la  w_0=0, \msP_{i,s}\bfone_\la w_0=P_{i,s}\bfone_\la w_0,\; \ti\ttk_i \bfone_\la w_0=\vep^{\mu_i}\bfone_\la w_0,\;\text{ and }\;\bigg[{\ti\ttk_i;0\atop l}\bigg] \bfone_\la w_0=\bigg[{\mu_i\atop l}\bigg]_\vep \bfone_\la w_0.$$
On the other hand, by \eqref{bfone} we have  $\ti\bffkk_{i}\bfone_\la=\vep^{\la_i-\la_{i+1}}\bfone_\la$ and $\big[{\ti\bffkk_{i};0\atop l}\big]\bfone_\la=\big[{\la_i-\la_{i+1}\atop l}\big]_\vep\bfone_\la$, where $\ti\bffkk_{i}=\zeta_{r,\vep}(\ti\ttk_i)$ and $\big[{\ti\bffkk_{i};0\atop l}\big]=\zeta_{r,\vep}(\big[{\ti\ttk_{i};0\atop l}\big])$. This implies that $\ti\ttk_i\bfone_\la w_0=\vep^{\la_i-\la_{i+1}}\bfone_\la
w_0$ and $\big[{\ti\ttk_{i};0\atop l}\big]_\vep\bfone_\la w_0=\big[{\la_i-\la_{i+1}\atop l}\big]_\vep\bfone_\la w_0$. So by Corollary \ref{m=m'} we have $\la_i-\la_{i+1}=\mu_i$, for $1\leq i\leq n-1$. This shows that $\la\in\La^+(n,r)$. Clearly, there
is a surjective $\afUnvep$-module homomorphism $\vi:\Mcp_\vep(\bfP,\la)\ra
V$ defined by sending $\bar u$ to $u\bfone_\la w_0$, for all $u\in\afUnvep$. Therefore, $V\cong \Vcp_\vep(\bfP,\la)$, as desired.
\end{proof}

According to Lemma \ref{Lem1 for classification Uvtg vep(n,r)} and Lemma \ref{Lem2 for classification Uvtg vep(n,r)} we obtain the following classification theorem of finite dimensional irreducible $\afUnrvep$-modules.
\begin{Thm}\label{classification Uvtg vep(n,r)}
The set $$\{\bar V_\vep(\bfP,\la)\mid\bfP\in\Pn,\,\la\in\La^+(n,r),\,\la_i-\la_{i+1}=\deg P_i(u)\,for\,1\leq i\leq n-1\}$$
is a complete set of nonisomorphic finite dimensional irreducible $\afUnrvep$-modules.
\end{Thm}

\subsection{Classification of finite dimensional irreducible $\afSrvep$-modules}

Combining Proposition \ref{prop1 of sGvep}, Proposition \ref{prop2 of sGvep} with Theorem \ref{classification Uvtg vep(n,r)} yields the following classification theorem of finite dimensional irreducible $\afSrvep$-modules.

\begin{Thm}\label{classification afSrvep}
The set $$\{V_\vep(\bfQ)\mid\bfQ\in\Qnr\}$$ is a complete set of nonisomorphic finite dimensional irreducible $\afSrvep$-modules.
\end{Thm}
\begin{proof}
We choose $N$ such that $N>\max\{n,r\}$.
Since $N>r$, by \cite[8.2]{Lu99} we have $\afSNrvep=\afUNrvep$.
It follows that $V_\vep(\bfQ)$ is a homomorphic image of $\bar M_\vep(\bfP,\la)$ for  $\bfQ\in\QNr$, where $P_i(u)=Q_i(\ttv^{i-1}u)/Q_{i+1}(\ttv^{i+1}u)$ for $1\leq i\leq N-1$ and $\la=\deg\bfQ$. This implies that $$V_\vep(\bfQ)\cong \Vcp_\vep(\bfP,\la)$$ as an $\afSNrvep$-module. Thus by Theorem \ref{classification Uvtg vep(n,r)} the set $$\{ V_\vep(\bfQ)\mid\bfQ\in\QNr\}$$
is a complete set of nonisomorphic finite dimensional irreducible $\afSNrvep$-modules. This together with \cite[6.2(g)]{Gr80} implies that the set $\{\sGvep(V_\vep(\bfQ))\not=0\mid\bfQ\in\QNr\}$ forms a complete set of nonisomorphic irreducible $\afSrvep$-modules. Now the assertion follows from Proposition \ref{prop1 of sGvep} and Proposition \ref{prop2 of sGvep}.
\end{proof}
\begin{Rems}
(1) If $\vep=1$, then $\afSrvep$ is the affine Schur algebra over $\mbc$. Therefore, we obtain a classification theorem for affine Schur algebras over $\mbc$ by Theorem \ref{classification afSrvep}.

(2) Assume $n\geq r$. Let $\varpi=(\varpi_1,\cdots,\varpi_n)\in\Lanr$, where $\varpi_i=1$ if $1\leq i\leq r$, and $\varpi_i=0$ otherwise. Then $\bfone_\varpi\afSrvep\bfone_\varpi\cong\afHrvep$. According to Theorem \ref{classification afSrvep} and \cite[6.2(g)]{Gr80}, the set $$\{\bfone_\varpi V_\vep(\bfQ)\not=0\mid\bfQ\in\Qnr\}$$
forms a complete set of  nonisomorphic finite dimensional irreducible
$\afHrvep$-modules. It would be interesting to determine the necessary and sufficient conditions for $\bfone_\varpi V_\vep(\bfQ)$ to be nonzero.

In the non root of unity case, finite dimensional irreducible modules for affine Hecke algebras of type $A$ were classified in terms of  multisegments (see \cite{BZ,Zelevinsky,Rogawski}). Finite dimensional irreducible modules for affine Hecke algebras of type $A$ at roots of unity
were classified in terms of
the aperiodic multisegments (see \cite{CG,Ariki,AM}).
The classification of irreducible modules for affine Hecke algebras in all types
were investigated in \cite{KL87,Xi94,Xi07}. It should be interesting to determine the aperiodic multisegments corresponding to $\bfone_\varpi V_\vep(\bfQ)$ when $\bfone_\varpi V_\vep(\bfQ)$ is nonzero.

(3) Following \cite[7.2]{FM}, let $\afUglZ$ be the $\mpZ$-subalgebra of $\afUglv$ generated by
$(\ttx_{i,s}^\pm)^{(m)}$, $\ttk_j^{\pm 1}$,
$\big[{\ttk_j;0\atop t}\big]$, $\ms Q_{i,s}$ ($1\leq i\leq n-1$, $1\leq j\leq n$, $m,t\in\mbn$, $s\in\mbz$). Let $\afUglvep=\afUglZ\ot_\mpZ\mbc$.

Let $\Pi$ be the set of polynomials $Q(u)\in\mbc[u]$ such that the constant term of $Q(u)$ is $1$. Let
$\Qn_\vep$ be the set of $n$-tuple of polynomials $(Q_1(u),\cdots,Q_n(u))$ such that $Q_i(u)\in\Pi$ and ${Q_j(u)}/{Q_j(u\vep^2)}\in\mbc[u]$ for
$1\leq i\leq n$ and $1\leq j\leq n-1$.
According to \cite[8.2]{FM}, finite dimensional polynomial irreducible $\afUglvep$-modules are indexed by the set $\Qn_\vep$.
Let
$$\Qnr_\vep=\bigg\{(Q_1(u),\cdots,Q_n(u))\in\Qn_\vep
\,\big|\,\sum_{1\leq i\leq n}\deg Q_i(u)=r\bigg\}.$$
Then the natural map $\ga:\Qnr\ra\Qnr_\vep$ defined by sending $(Q_1(u),\cdots,Q_n(u))$ to $(Q_1(u)|_{v=\vep},
\cdots,Q_n(u)|_{v=\vep})$ is bijective. Therefore, polynomial representation of $\afUglvep$  should be closely related to representation of $\afSrvep$.
\end{Rems}

\subsection{Morita equivalences of affine quantum Schur algebras}
We are now ready to establish Morita equivalences between affine quantum Schur algebras.
Before proving Proposition \ref{category equivalence of affine quantum Schur algebras}  we shall need some preliminary lemmas.

Let $S$ be an associative algebra over a field $k$ and let
$\mpf\not=0$ be any idempotent in $S$. We define the functors $\sF$ and $\sH$ as follows:
\begin{equation*}
\begin{split}
\sF:S&\hmod\lra\mpf S\mpf\hmod,\qquad
V\longmapsto \mpf V,\\
\sH:\mpf S\mpf&\hmod\lra S \hmod,\qquad
W\longmapsto S\mpf \ot_{\mpf S\mpf}W.
\end{split}
\end{equation*}
\begin{Lem}\label{S mpf V}
Assume $\sF(L)\not=0$ for all finite dimensional irreducible $S$-module $L$.
Then $S\mpf V=V$ for $V\in S\hmod$.
\end{Lem}
\begin{proof}
Let $V$ be a finite dimensional $S$-module and let $0=V_t\han V_{t-1}\han\cdots\han V_1\han V_0=V$ be the composition series of $V$.
For $0\leq i\leq t-1$, there is a natural $S$-module homomorphism   $\vi_i:S\mpf V_i/S\mpf V_{i+1}\ra V_i/V_{i+1}$ defined by sending $\bar x$ to $\bar x$ for $\bar x\in S\mpf V_i/S\mpf V_{i+1}$. Since $\sF (V_i/V_{i+1})\not=0$ we have $\text{Im}(\vi_i)\not=0$.
It follows that $\vi_i$ is surjective since $V_i/V_{i+1}$ is irreducible. This implies that
$\dim S\mpf V=\sum_{0\leq i\leq t-1}\dim(S\mpf V_i/S\mpf V_{i+1})\geq\sum_{0\leq i\leq t-1}\dim(V_i/V_{i+1})=\dim V$. Consequently, we conclude that $S\mpf V=V$.
\end{proof}

\begin{Lem}\label{alpha V}
Assume $\sF(L)\not=0$ for all finite dimensional irreducible $S$-module $L$.
Then for each $V\in S\hmod$, the $S$-module homomorphism $\alpha_V:\sH\circ\sF(V)\ra V$ defined by sending $x \ot \mpf w$ to $x\mpf w$ is an isomorphism for $x\in S\mpf$ and $w\in V$.
\end{Lem}
\begin{proof}
We proceed by induction on $\ell(V)$, where $\ell(V)$ is the length of $V$. According to \cite[(6.2f)]{Gr80} we see that
$\al_V$ is an isomorphism in the case where $V$ is irreducible.
Assume now that $\ell(V)>1$. Using Lemma \ref{S mpf V} we see that $\al_V$ is surjective. Thus it is enough to prove that $\dim(\sH\circ\sF(V))=\dim V$.
Let $M$ be a maximal $S$-submodule of $V$. Then $\al_M$ is an isomorphism by induction.
Let $\iota$ be the inclusion map from $M$ to $V$.  Since $\iota\circ\al_M=\al_V\circ(\sH\circ\sF(\iota))$ and $\al_M$ is injective we conclude that the map $\sH\circ\sF(\iota)$ is injective. Thus we may regard $\sH\circ\sF(M)$ as a submodule of $\sH\circ\sF(V)$. As $V/M$ is irreducible, we see that $\sH\circ\sF(V)/\sH\circ\sF(M)\cong\sH\circ\sF(V/M)\cong V/M$. This together with the fact that $\sH\circ\sF(M)\cong  M$ shows that $\dim(\sH\circ\sF(V))=\dim V$. The assertion follows.
\end{proof}

\begin{Coro}\label{category equivalence}
Assume $\sF(L)\not=0$ for all finite dimensional irreducible $S$-module $L$.
Then the functor $\sF:S\hmod\ra\mpf S\mpf\hmod$ is an equivalence of categories.
\end{Coro}
\begin{proof}
According to \cite[6.2d]{Gr80} we see that $\sF\circ\sH\cong 1_{\mpf S
\mpf\hmod}$ where $1_{\mpf S
\mpf\hmod}$ is the identity functor on $\mpf S
\mpf\hmod$.  Now the assertion follows from Lemma \ref{alpha V}.
\end{proof}

Combining Proposition \ref{prop1 of sGvep}, Theorem \ref{classification afSrvep} with Corollary \ref{category equivalence} yields the following Morita equivalences of affine quantum Schur algebras, which is the affine version of \cite[(6.5g)]{Gr80}.
\begin{Prop}\label{category equivalence of affine quantum Schur algebras}
Assume $N\geq n\geq r$. Then the functor $\sGvep:\afSNrvep\hmod\ra\afSrvep\hmod$
is an equivalence of categories.
\end{Prop}

\end{document}